\documentclass[12pt]{amsart}  

\usepackage[latin1]{inputenc}
\usepackage{amsmath} 
\usepackage{amsfonts}
\usepackage{amssymb}
\usepackage{stmaryrd}
\usepackage{latexsym} 
\usepackage{graphicx}
\usepackage{subfigure}
\usepackage{color}
\usepackage{hyperref}
\usepackage{verbatim}
\usepackage[all]{xy}
\usepackage{graphics}
\usepackage{pdfsync}

\theoremstyle{plain}
\newtheorem{theorem}{Theorem}[section]
\newtheorem{proposition}[theorem]{Proposition}

\newtheorem{lemma}[theorem]{Lemma}

\newtheorem{definition}[theorem]{Definition}
\newtheorem{example}[theorem]{Example}

\theoremstyle{remark}

\numberwithin{equation}{section}
\newcommand\ctranspose[1]{{#1}^t}

\newlength\cellsize 
\savebox2{%
\begin{picture}(15,15)
\put(0,0){\line(1,0){15}}
\put(0,0){\line(0,1){15}}
\put(15,0){\line(0,1){15}}
\put(0,15){\line(1,0){15}}
\end{picture}}
\newcommand\cellify[1]{\def\thearg{#1}\def\nothing{}%
\ifx\thearg\nothing
\vrule width0pt height\cellsize depth0pt\else
\hbox to 0pt{\usebox2\hss}\fi%
\vbox to 15\unitlength{
\vss
\hbox to 15\unitlength{\hss$#1$\hss}
\vss}}
\newcommand\tableau[1]{\vtop{\let\\=\cr
\setlength\baselineskip{-16000pt}
\setlength\lineskiplimit{16000pt}
\setlength\lineskip{0pt}
\halign{&\cellify{##}\cr#1\crcr}}}
\savebox3{%
\begin{picture}(15,15)
\put(0,0){\line(1,0){15}}
\put(0,0){\line(0,1){15}}
\put(15,0){\line(0,1){15}}
\put(0,15){\line(1,0){15}}
\end{picture}}
\newcommand\expath[1]{%
\hbox to 0pt{\usebox3\hss}%
\vbox to 15\unitlength{
\vss
\hbox to 15\unitlength{\hss$#1$\hss}
\vss}}
\newcommand\bas[1]{\omit \vbox to \cellsize{ \vss \hbox to \cellsize{\hss$#1$\hss} \vss}}

\begin{document}

\title[Chromatic classical symmetric functions]{Chromatic classical symmetric functions}

\author{Soojin Cho}
\address{
 Department of Mathematics,
 Ajou University,
 Suwon  443-749, Korea}
\email{chosj@ajou.ac.kr}

\author{Stephanie van Willigenburg}
\address{
 Department of Mathematics,
 University of British Columbia,
 Vancouver BC V6T 1Z2, Canada}
\email{steph@math.ubc.ca}

\thanks{
The first author was supported by the Basic Science Research Program through the National Research Foundation of Korea (NRF) funded by the Ministry of Education (NRF-2015R1D1A1A01057476).  The second author was supported in part by the National Sciences and Engineering Research Council of Canada.}

\begin{abstract}
In this note we classify when a skew Schur function is a positive linear combination of power sum symmetric functions. We then use this to determine precisely when any scalar multiple of a skew Schur function is the chromatic symmetric function of some graph. From here we are able to prove that of the classical bases for symmetric functions only certain scalar multiples of the elementary symmetric functions can be realised as the chromatic symmetric function of some graph, namely a particular union of complete graphs.
\end{abstract}

\maketitle

\section{Introduction}\label{sec:intro}

The areas of symmetric functions and chromatic functions both have a long and distinguished history. More recently  a chromatic symmetric function has been defined for any finite simple graph \cite{Stan95}. It naturally generalizes the chromatic function and its positivity when expanded into elementary symmetric functions, for example \cite{GebSag, GP, SW}, or Schur functions, for example \cite{Gasharov}, has been a keen avenue of research due to connections to other areas such as representation theory and algebraic geometry. Other avenues in this area include sufficient conditions for chromatic symmetric function equality \cite{Orellana} and the recent discovery that chromatic symmetric functions give rise to infinitely many bases for the algebra of symmetric functions \cite{ChovW}. Regarding the latter, the question has consequently been frequently raised whether the classical bases for the algebra of symmetric functions can be realised as chromatic symmetric functions, especially the elementary symmetric functions or Schur functions. In this note we answer this question comprehensively (Theorem~\ref{the:bases}) after first establishing which scalar multiples of skew Schur functions can be realised as chromatic symmetric functions (Theorem~\ref{the:skew}) that is reliant on establishing when they expand positively in terms of power sum symmetric functions (Proposition~\ref{prop:ppositive}).

\section{Chromatic classical bases}\label{sec:chromatic}

Before we prove our results, for completeness, we first recall classical combinatorial and algebraic concepts, which will be necessary later. Readers familiar with the subject may safely skip to Definition~\ref{def:XG}. A \emph{partition} $\lambda = (\lambda _1 , \ldots , \lambda _\ell)$ of $n$ is a list of weakly decreasing positive integers, or \emph{parts}, $\lambda _i$ whose sum is $n$, which we denote by $\lambda \vdash n$. We denote by $\emptyset$ the empty partition whose sum is 0. Given a partition $\lambda = (\lambda _1 , \ldots , \lambda _\ell)$ its corresponding \emph{diagram}, also denoted by $\lambda$, is the array of left-justified boxes with $\lambda _i$ boxes in row $i$, where we number the rows from top to bottom and the columns from left to right. Given two partitions $\lambda = (\lambda _1 , \ldots , \lambda _\ell)$ and $\mu = (\mu _1 , \ldots , \mu _k)$ where $ k\leq \ell$ and $\mu _i \leq \lambda _i$ for $1\leq i \leq k$ we say the \emph{skew diagram} $\lambda / \mu$ is the array of boxes
$$\lambda/\mu = \{ (i,j) \mid (i,j) \in \lambda \mbox{ and } (i,j) \not \in \mu \}.$$Note that if $\mu = \emptyset$ then $\lambda / \mu$ is the diagram $\lambda$. The number of boxes in $\lambda / \mu$ is called its \emph{size} and is denoted by $|\lambda / \mu|$. The \emph{transpose} of the skew diagram $\lambda / \mu$ is the array of boxes
$$\ctranspose{(\lambda/\mu)}  = \{ (j,i) \mid (i,j) \in \lambda \mbox{ and } (i,j) \not \in \mu \}.$$Two skew diagrams that will be of particular interest to us are horizontal and vertical strips. We say a skew diagram is a \emph{horizontal strip} if no two boxes lie in the same column, and is a \emph{vertical strip} if no two boxes lie in the same row. For simplicity we will often denote a skew diagram by $D$.

Given a skew diagram $\lambda / \mu$ we say that a \emph{semistandard Young tableau (SSYT)} of \emph{shape} $\lambda / \mu$ is a filling of the boxes of $\lambda / \mu$ with positive integers such that the entries in each row weakly increase when read from left to right, and the entries in each column strictly increase when read from top to bottom. Given an SSYT, $T$, with largest entry $max _T$ we define the \emph{content} of $T$ to be
$$c(T) = (c_1(T),   \ldots , c_{max _T}(T))$$where $c_i(T)$ is the number of times $i$ appears in $T$.

\begin{example}\label{ex:tabs}
Below on the left is the skew diagram $D=(6,4,4,1)/(3,2)$ and on the right its transpose $\ctranspose{D}$.
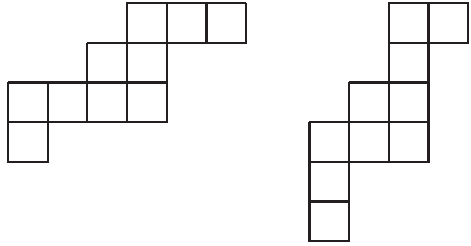
$$\tableau{&&&\ &\ &\ \\
&&\ &\ \\
\ &\ &\ &\ \\\ } \qquad 
\tableau{&&\ &\ \\
 &&\ \\
&\ &\ \\
\ &\ &\ \\
\ \\
\ }$$Note that $D$ is neither a horizontal nor a vertical strip. Below is an SSYT of shape $D$ and content $(5,3,1,1)$.
$$\tableau{&&&1 &1 &1 \\
&&1 &2 \\
1 &2 &2 &3 \\4 }$$
\end{example}

Given partitions and tableaux we can now turn our attention to symmetric functions, which will be the focus of our study, and whose algebra is a subalgebra of $\mathbb{C} [[ x_1, x_2, \ldots ]]$. Given a partition $\lambda = (\lambda _1, \ldots , \lambda _\ell)\vdash n \geq 1$ we define the \emph{monomial symmetric function} $m_\lambda$ to be
$$m _\lambda = \sum x^{\lambda _1} _{i_1} \cdots x^{\lambda _\ell} _{i_\ell}$$where the sum is over all $\ell$-tuples $(i_1, \ldots , i_\ell)$ of distinct indices that yield distinct monomials.  We define the $r$-th \emph{power sum symmetric function} to be
$p_r=m_{(r)}$ from which we define the \emph{power sum symmetric function} $p_\lambda$ to be
$$p_\lambda = p_{\lambda _1} \cdots p_{\lambda _\ell}.$$Meanwhile, we define the renowned \emph{Schur function} $s_\lambda$ to be
$$s_\lambda = \sum _{\mu \vdash n} K_{\lambda\mu} m_\mu$$where $K_{\lambda\mu}$ is the number of SSYTs of shape $\lambda$ and content $\mu$. Schur functions can further be generalised to \emph{skew Schur functions} $s_{D}$ for a skew diagram $D$
\begin{equation}\label{eq:sasskew}s_{D} = \sum _{\nu \vdash |D|} c  _{D\nu} s_\nu\end{equation}where $c  _{D\nu}$ is the number of SSYTs, $T$, of shape $D$ and content $\nu$ such that when the entries of $T$ are read from right to left from the top row to the bottom row, the number of $i$'s that have been read is always weakly greater than the number of $(i+1)$'s that have been read, for all $i\geq 1$. Note that the SSYT in Example~\ref{ex:tabs} satisfies this latter criterion. This rule for computing $c  _{D\nu}$ is called the \emph{Littlewood-Richardson rule}. Note that when $D = \lambda /\emptyset$ the Schur function $s_\lambda$ is recovered. We define the $r$-th \emph{complete homogeneous symmetric function} to be $h_{r}= s_{(r)}$, and
 if $D$ is a horizontal strip with $\alpha _i$ boxes in row $i$ for $1\leq i \leq \ell$ such that $\alpha _1, \ldots , \alpha _\ell$ in weakly decreasing order yields $\lambda _1 , \ldots , \lambda _\ell$ then  we define the \emph{complete homogeneous symmetric function} $h_\lambda$ to be
\begin{equation}\label{eq:hasskew}h_\lambda = s_D = h_{\alpha _1}\cdots h_{\alpha _\ell}= h_{\lambda _1}\cdots h_{\lambda _\ell}.\end{equation}We define the $r$-th \emph{elementary symmetric function} to be $e_{r}= s_{(1^r)}$ 
where $(1^r)$ is the partition consisting of $r$ parts equal to 1. Furthermore, $\ctranspose{D}$ is a vertical strip with $\alpha _i$ boxes in column $i$ for $1\leq i \leq \ell$ and we define the \emph{elementary symmetric function} $e_\lambda$ to be
\begin{equation}\label{eq:easskew}e_\lambda = s_{\ctranspose{D}} = e_{\alpha _1}\cdots e_{\alpha _\ell} = e_{\lambda _1}\cdots e_{\lambda _\ell}.\end{equation}Lastly, by convention
$$e_\emptyset = h_\emptyset = m_\emptyset = p_\emptyset = s_\emptyset = 1.$$Note also that
$$e_{(1)} = h_{(1)} = m_{(1)} = p_{(1)} = s_{(1)}$$and$$e_{(1^r)} = h_{(1^r)} =  p_{(1^r)}\qquad e_{(r)}=m_{(1^r)}=s_{(1^r)}$$for $r \geq 1$. The \emph{algebra of symmetric functions} is then the graded algebra
$$\Lambda = \Lambda ^0 \oplus \Lambda ^1 \oplus \cdots$$where a basis for the $n$-th graded part is given by any of $\{e_\lambda \} _{\lambda \vdash n}$, $\{h_\lambda \} _{\lambda \vdash n}$, $\{m_\lambda \} _{\lambda \vdash n}$, $\{p_\lambda \} _{\lambda \vdash n}$, $\{s_\lambda \} _{\lambda \vdash n}$. There exists a scalar product on $\Lambda$ such that the power sum symmetric functions form an orthogonal basis, that is, for partitions $\lambda$ and $\mu$
$$\langle p_\lambda , p_\mu \rangle = \delta _{\lambda\mu}z_\lambda$$while the Schur functions form an orthonormal basis, that is,
$$\langle s_\lambda , s_\mu \rangle = \delta _{\lambda\mu}$$where $\delta _{\lambda\mu}=1$ if $\lambda = \mu$ and 0 otherwise, and  $z_\lambda$ is a positive integer reliant on $\lambda$.
A further useful tool is the involution $\omega$ that satisfies$$\omega(s_D) = s_{\ctranspose{D}}$$for any skew diagram $D$. One last useful fact is that
\begin{equation}\label{eq:hrasp}
h_r=\sum _{\lambda \vdash r} z_\lambda ^{-1} p_\lambda .
\end{equation}

The symmetric functions above are not the only ones of interest to us. In \cite{Stan95} Stanley defined the chromatic symmetric function as follows, which is reliant on a graph that is finite and simple and the notion of a proper colouring, where given a graph $G$ with vertex set $V$ a proper colouring $\kappa$ of $G$ is a function
$$\kappa: V \rightarrow \{1,2,\ldots\}$$such that if $v_1, v_2 \in V$ are adjacent then $\kappa(v_1)\neq \kappa(v_2)$.

\begin{definition}\label{def:XG}  For a finite simple graph $G$ with vertex set $V=\{v_1, \dots, v_n\}$ and edge set $E$ the \emph{chromatic symmetric function of $G$} is defined to be
$$ X_G=\sum_{\kappa} x_{\kappa(v_1)}   \cdots x_{\kappa(v_n)}$$where the sum is over all proper colourings $\kappa$ of $G$.
\end{definition}

The chromatic symmetric function can equivalently be defined in terms of power sum or monomial symmetric functions with the aid of some natural partitions. Given a graph $G$ with vertex set $V= \{ v_1, \ldots , v_n\}$, edge set $E$, and a subset $S\subseteq E$, let $\lambda(S)$ be the partition of $n$ whose parts are equal to the number of vertices in the connected components of the spanning subgraph of $G$ with vertex set $V$ and edge set $S$. Meanwhile, a stable partition $\pi$ of the vertices of $G$ is a partition of $v_1, \ldots , v_n$ such that each block is totally disconnected. Let $\mu(\pi)$ be the partition of $n$ whose parts are equal to the number of vertices in each block of $\pi$.

\begin{lemma}\label{lem:pm}\cite[Proposition 2.4, Theorem 2.5]{Stan95} For a finite simple graph $G$ with vertex set $V$ and edge set $E$ we have the following.
\begin{enumerate}
\item $X_G = \sum _{S\subseteq E} (-1) ^ {|S|} p _{\lambda(S)}.$
\item $X_G = \sum _\pi (\pi _1 ! \pi _2 ! \cdots ) m_{\mu(\pi)}$ where the sum is over all stable partitions $\pi$ of $G$, and $\pi _i$ is the multiplicity of $i$ in $\mu(\pi)$.
\end{enumerate}
\end{lemma}

The chromatic symmetric function also satisfies the following two crucial properties, recalling that when we say a function is \emph{$p$-positive} (respectively, \emph{-negative}) we mean it is a positive (respectively, negative) linear combination of power sum symmetric functions.

\begin{lemma}\label{lem:prodrule} \cite[Proposition 2.3]{Stan95} If a finite simple graph $G$ is a disjoint union of subgraphs $G_1 \cup \cdots \cup G_\ell$ then $X_G = \prod _{i=1} ^\ell X_{G_i}$.
\end{lemma}

\begin{lemma}\label{lem:omega}\cite[Corollary 2.7]{Stan95} For any finite simple graph $G$, $\omega(X_G)$ is $p$-positive.
\end{lemma}

We are now close to establishing which classical symmetric functions are also the chromatic symmetric function of some graph. To this end we first deduce a natural result on skew Schur functions when expanded as a sum of power sum symmetric functions.

\begin{proposition}\label{prop:ppositive} A skew Schur function $s_D$ is $p$-positive if and only if $D$ is a horizontal strip. \end{proposition}

\begin{proof} Let $D$ be a horizontal strip with $\ell$ rows and $\alpha _i$ boxes in row $i$ for $1\leq i \leq \ell$. Then by Equation~\eqref{eq:hasskew} 
$$s_D = h _{\alpha _1}\cdots h_{\alpha _\ell}.$$We know by Equation~\eqref{eq:hrasp} that $h_r$ is $p$-positive, and hence so is 
$$h _{\alpha _1}\cdots h_{\alpha _\ell} = s_D.$$

Conversely, let $D$ contain two boxes in some column, $|D|=n$, and
$$s_D = \sum _{\lambda \vdash n} a _{D\lambda } p _\lambda .$$By the Littlewood-Richardson rule of Equation~\eqref{eq:sasskew} we have that the coefficient of $s_{(n)}$ in $s_D$ is 0, that is,
\begin{equation}\label{eq:sDsn}\langle s_D, s_{(n)} \rangle = 0.\end{equation}By Equation~\eqref{eq:hrasp} we know that
$$s_{(n)} = h_n=\sum _{\lambda \vdash n} z_\lambda ^{-1} p_\lambda$$and substituting into Equation~\eqref{eq:sDsn} we get
$$\langle \sum _{\lambda \vdash n} a _{D\lambda } p _\lambda , \sum _{\lambda \vdash n}z_\lambda ^{-1} p_\lambda \rangle = 0.$$Since the power sum symmetric functions are an orthogonal basis for $\Lambda$, we have that $a_{D\lambda } \not \geq 0$ for all $\lambda \vdash n$, and the result follows.
\end{proof}

With this proposition we now classify which skew Schur functions can be realised as the chromatic symmetric function of some graph, and key to our classification is the complete graph $K_n$, $n\geq 1$ consisting of $n$ vertices, each pair of which are adjacent.

\begin{theorem}\label{the:skew} $X_G = cs_D$ for some graph $G$ and $c\neq 0$ if and only if $D$ is a vertical strip. In particular, if $D$ is a vertical strip consisting of $\ell$ columns with $\alpha _i$ boxes in column $i$ for $1\leq i \leq \ell$ then
$$X_G = \left( \prod _{i=1} ^\ell \alpha _i !\right) s_D$$where
$$G=K_{\alpha _1} \cup \cdots \cup K _{\alpha _\ell}.$$
\end{theorem}

\begin{proof} Let $D$ be a vertical strip with $\ell$ columns containing $\alpha _i$ boxes in column $i$ for $1\leq i \leq \ell$. Then, since for $K_n$, $n\geq 1$ we have that \cite[Theorem 8]{ChovW}
$$X_{K_n} = n! e_n,$$by Equation~\eqref{eq:easskew} and Lemma~\ref{lem:prodrule} we obtain
$$s_D=\prod _{i=1} ^{\ell} e _{\alpha _i} = \prod _{i=1} ^{\ell} \frac{1}{\alpha _i !} X_{K_{\alpha _i}} = \left(\prod _{i=1} ^{\ell} \frac{1}{\alpha _i !}\right) X_{K_{\alpha _1} \cup \cdots \cup K_{\alpha _\ell}}.$$

Now let $D$ contain two boxes in some row, and assume that there exists a graph $G$ such that $X_G = cs_D$ for $c\neq 0$. Then by Lemma~\ref{lem:omega}
$$\omega(X_G)=\omega (cs_D) = cs_{\ctranspose{D}}$$is $p$-positive. Hence $s_{\ctranspose{D}}$ is either $p$-positive or $p$-negative and moreover $\ctranspose{D}$ contains two boxes in some column. By Proposition~\ref{prop:ppositive} this is not possible, and hence no such $G$ exists.
\end{proof}
We are now ready to establish which classical symmetric functions can be realised as the chromatic symmetric function of some graph.

\begin{theorem}\label{the:bases}Of the classical symmetric functions $\{e_\lambda \} _{\lambda \vdash n\geq 1}$, $\{h_\lambda \} _{\lambda \vdash n\geq 1}$, $\{m_\lambda \} _{\lambda \vdash n\geq 1}$, $\{p_\lambda \} _{\lambda \vdash n\geq 1}$ and $\{s_\lambda \} _{\lambda \vdash n\geq 1}$ only those that are scalar multiples of the elementary symmetric functions $\{e_\lambda \} _{\lambda \vdash n\geq 1}$ can be realised as the chromatic symmetric function of some graph. In particular, if $\lambda = (\lambda _1, \ldots, \lambda _\ell) \vdash n \geq 1$ then
$$e_\lambda= \left(\prod _{i=1} ^{\ell} \frac{1}{\lambda _i !}\right) X_{K_{\lambda _1} \cup \cdots \cup K_{\lambda _\ell}}.$$
\end{theorem}

\begin{proof} Since $e_{(1^n)} = h_{(1^n)} = p_{(1^n)}$ for $n\geq 1$ we consider this to be $e_{(1^n)}$ for convenience, and likewise consider $e_{(n)}=m_{(1^n)}=s_{(1^n)}$ for $n\geq 1$ to be $e_{(n)}$. The result for $\{p_\lambda \} _{\lambda \vdash n\geq 1}$ and $\{m_\lambda \} _{\lambda \vdash n\geq 1}$ follows immediately from Lemma~\ref{lem:pm}. Meanwhile, the result for  $\{s_\lambda \} _{\lambda \vdash n\geq 1}$ and $\{h_\lambda \} _{\lambda \vdash n\geq 1}$ follows from Equations~\eqref{eq:sasskew} and \eqref{eq:hasskew} combined with Theorem~\ref{the:skew}. Equation~\eqref{eq:easskew} combined with Theorem~\ref{the:skew} then yields the result for $\{e_\lambda \} _{\lambda \vdash n\geq 1}$.
\end{proof}

\section*{Acknowledgements}\label{sec:acknow} The authors would like to thank Richard Stanley and Vasu Tewari for helpful conversations, the referee, and Ajou University where some of the research took place.




\begin{thebibliography}{10}


\bibitem{ChovW} 
\textsc{Cho, S.} and \textsc{van Willigenburg, S.} (2016).
Chromatic bases for symmetric functions. 
\textit{Electron. J. Combin.} 
\textbf{23} P1.15 6pp.

\bibitem{Gasharov} 
\textsc{Gasharov, V.} (1996).
Incomparability graphs of $(3 + 1)$-free posets are
$s$-positive. 
\textit{Discrete Math.} 
\textbf{157} 193--197.

\bibitem{GebSag} 
\textsc{Gebhard, D.} and \textsc{Sagan, B.}  (2001).
A chromatic symmetric function in noncommuting variables.
\textit{J. Algebraic Combin.} 
\textbf{13} 227--255. 

\bibitem{GP} 
\textsc{Guay-Paquet, M.} 
A modular law for the chromatic symmetric functions of $(3+1)$-free posets. 
{\tt arXiV:1306.2400v1}

\bibitem{Orellana} 
\textsc{Orellana, R.} and \textsc{Scott, G.} (2014).
Graphs with equal chromatic symmetric function. 
\textit{Discrete Math.} 
\textbf{320} 1--14.

\bibitem{SW} 
\textsc{Shareshian, J.} and \textsc{Wachs, M.} 
Chromatic quasisymmetric functions. 
\textit{Adv. Math.} 
to appear.

\bibitem{Stan95} 
\textsc{Stanley, R.} (1995).
A symmetric function generalization of the chromatic polynomial of a graph.
\textit{Adv. Math.} 
\textbf{111} 166--194.

\end{thebibliography}
\end{document}